\documentclass[10pt,a4paper]{amsart}

\usepackage{setspace}
\usepackage{amsfonts}
\usepackage{amsmath}
\usepackage{amsthm}

\usepackage[colorlinks=true,linkcolor=red,citecolor=blue,urlcolor=blue]{hyperref} 

\usepackage{graphicx}
\usepackage{stmaryrd}
\usepackage{array}

\hypersetup{
  pdftitle   = {},
  pdfauthor  = {},
  pdfcreator = {\LaTeX\ with package \flqq hyperref\frqq}
}

\DeclareMathAlphabet{\mathpzc}{OT1}{pzc}{m}{it}

\newtheorem{theorem}{Theorem}[section]
\newtheorem*{theorem*}{Theorem}

\newtheorem{lemma}[theorem]{Lemma}
\newtheorem*{lemma*}{Lemma}
\newtheorem{corollary}[theorem]{Corollary}

\newtheorem*{conjecture*}{Conjecture}

\theoremstyle{definition}
\newtheorem{definition}[theorem]{Definition}

\theoremstyle{remark}
\newtheorem{remark}[theorem]{Remark}

\newcommand{\disp}{\displaystyle}
\newcommand{\R}{\mathbb{R}}
\newcommand{\C}{\mathbb{C}}

\newcommand{\g}{\mathfrak{g}}
\newcommand{\h}{\mathfrak{h}}
\newcommand{\m}{\mathfrak{m}}

\newcommand{\gs}{\mathfrak{g}_\sigma}
\newcommand{\mt}{\mathfrak{m}_\Theta}
\newcommand{\ul}{\mathfrak{u}}
\newcommand{\kt}{\mathfrak{k}_\Theta}

\newcommand{\qt}{\mathfrak{q}_\Theta}
\newcommand{\gal}{\mathfrak{g}_\alpha}

\newcommand{\su}{\mathfrak{su}}
\newcommand{\sll}{\mathfrak{sl}}
\newcommand{\Pp}{\mathcal{P}}

\newcommand{\flag}{\mathbb{F}_\Theta}

\newcommand{\Kt}{K_{\Theta}}

\newcommand{\piteta}{\Pi(\Theta)}
\newcommand{\Lamjv}{\Lambda^{\#}}
\newcommand{\lamjv}{\lambda_\sigma^\#}

\newcommand{\sve}{I_{\Lambda}^{\phi}(q)}

\newcommand{\tipoum}{\text{SO}(2(m+k)+1)/(\text{U}(k) \times \text{SO}(2m+1))}
\newcommand{\tipodois}{\text{Sp}(m+k)/(\text{U}(m) \times \text{Sp}(k))}
\newcommand{\ft}{\text{SU}(3)/T^2}

\DeclareMathOperator{\Ad}{Ad}

\numberwithin{equation}{section}

\begin{document}

\title{On the stability of harmonic maps under the homogeneous Ricci flow}

\author{Rafaela F. do Prado and Lino Grama}
\address{Department of Mathematics - IMECC, University of Campinas - Brazil}
\email{linograma@gmail.com, rafaelafprado@gmail.com}
\thanks{LG is supported by Fapesp grant no. 2014/17337-0 and 2012/18780-0. RP is
supported by CNPq grant 142259/2015-2.}
\begin{abstract}
In this work we study properties of stability and non-stability of harmonic maps under the homogeneous Ricci flow. We provide examples where the stability (non-stability) is preserved under the Ricci flow and an example where the Ricci flow does not preserve the stability of an harmonic map.
\end{abstract}

\maketitle

\section{Introduction}

Let $(M,g)$ be a Riemannian manifold. The {\em Ricci flow} is a 1-parameter family of metrics $g(t)$ in $M$ with initial metric $g$ that satisfies the Ricci flow equation 

\begin{equation}
\label{RicciEquation}
\dfrac{\partial}{\partial t}g(t) = -2Ric(g(t)).
\end{equation}
The Ricci flow was first introduced by Hamilton based on the work of Eells-Sampson as pointed out by him in \cite{Ha82}. One of the main ideas is to start with any metric $g$ of strictly positive Ricci curvature and try to improve it by means of a heat equation. Similar methods were used by Eells-Sampson in the context of harmonic maps (see \cite{E-S}).
In the same work \cite{Ha82} Hamilton showed that positive Ricci curvature is preserved by (\ref{RicciEquation}) on closed 3-manifolds. Hamilton also proved that the same results hold for positive isotropic curvature in closed 4-manifolds \cite{Ha97}. However, some curvatures conditions may not be preserved by the Ricci flow. For example, B\"ohm and Wilking \cite{B-H} exhibited homogeneous metrics with sec $>0$ that develop mixed Ricci curvature in dimension 12, and mixed sectional curvature in dimension 6.  Abiev and Nikonorov  \cite{AN} proved that, for all Wallach spaces, the normalized Ricci flow evolves all generic invariant Riemannian metrics with positive sectional curvature into metrics with mixed sectional curvature and, more recently,  Bettiol and Krishnan \cite{BK} exhibit examples of closed 4-manifolds with nonnegative sectional curvature that develop mixed curvature under Ricci flow. There are several recent papers about Ricci flow (and other geometric flows) in homogeneous spaces, for example, \cite{AN}, \cite{GrMa09}, \cite{GrMa12}, \cite{lauret2}, \cite{lauret3} and references therein.


In this paper, we are interested in studying the stability of harmonic maps under the homogeneous Ricci flow. More specific, we want to know if the Ricci flow preserves stability of a class of harmonic maps from Riemann surfaces to homogeneous spaces. Since harmonic maps are critical points of the energy functional, we are interested in whether the second variation of the energy of these maps are positive or non-negative a certain variation. In this sense, we say that a harmonic map is stable if the second variation of the energy of this map is non-negative for every variation.

The harmonic maps we are going to consider are the so called generalized holomorhic-horizontal. They were first introduced by Bryant \cite{Bry85}. These maps are {\em equiharmonic}, that is, harmonic with respect to any invariant metric. Equiharmonic maps were introduced by Negreiros in \cite{caio} and several results about stability and non-stability of those kind of maps were proved in \cite{NeGrSM11}.

In \cite{GrMa12} and \cite{GrMa09} we have a study on the behavior of the homogeneous Ricci flow of left-invariant metrics on three types of homogeneous manifolds
\begin{equation} \label{t1}
\tipoum , 
\end{equation}
\begin{equation} \label{t2}
\tipodois ,
\end{equation}
 and 
\begin{equation} \label{t3}
 SU(3)/T^2
\end{equation}
 by a dynamical system point of view. We are interested in studying stability and non-stability of generalized holomorphic-horizontal maps in these three classes of homogeneous manifolds.

The homogeneous spaces described in (\ref{t1}), (\ref{t2}) and (\ref{t3}) belongs to a large class of homogeneous spaces called {\em generalized flag manifolds} and these spaces appear in several well known situations. For example: the family (\ref{t1}) includes the non-symmetric complex homogeneous space $\mathbb{CP}^{2n+1}=Sp(n+1)/(Sp(n)\times U(1))$ - the total space of a twistor fibration over $\mathbb{HP}^n$; the family (\ref{t2}) includes the Calabi twistor space $SO(2n+1)/U(n)$ used in the construction of harmonic maps from $S^2$ to $S^{2n}$; and the Wallach flag manifold $SU(3)/T^2$ is a 6-dimensional homogeneous space that admits invariant metric with positive sectional curvature.

By analyzing the dynamics of the homogeneous Ricci flow together with the results concerning stability/unstability of equiharmonic maps, we prove the following results:
\\
\\
\textbf{Theorem A} \textit{The homogeneous Ricci flow preserves the stability (respectively non-stability) of a generalized holomorphic-horizontal map on the homogeneous spaces $\tipoum$ and $\tipodois$.
}
\\
\\
\textbf{Theorem B} \textit{The homogeneous Ricci flow does not preserve stability of a generalized holomorphic-horizontal map on the homogeneous space $\ft$}.
\\

This paper is organized as follows. In section 1, we recall the main results about the geometry of flag manifolds. In section 2, we review some of the theory of holomorphic  maps on flag manifolds, including the results on whether a generalized holomorphic-horizontal map is stable or unstable. In section 3, we first recall the homogeneous Ricci flow of invariant metric on $\tipoum$, $\tipodois$ and $\ft$ and then prove our results.

\section{The geometry of generalized flag manifolds}
\subsection{Generalized flag manifolds}

Let $\mathfrak{g}$ be a complex semisimple Lie algebra and $G$ the correspondent Lie group. Let $\mathfrak{h}$ be a Cartan subalgebra of $\mathfrak{g}$ and denote by $\Pi $ the set of roots of $\left(\mathfrak{g},\mathfrak{h}\right) $. Then 

\begin{equation*}
\mathfrak{g}=\mathfrak{h}\oplus \sum_{\alpha \in \Pi }\mathfrak{g}_{\alpha },
\end{equation*}
where $\mathfrak{g}_{\alpha }=\{X\in \mathfrak{g};\,\forall H\in \mathfrak{h},\,[H,X]=\alpha (H)X\}$ denotes the root space (complex 1-dimensional).

Let $(\cdot ,\cdot) $ be the Cartan-Killing form of the Lie algebra $\mathfrak{g}$ and fix a Weyl basis of $\mathfrak{g}$, that is, choose vectors $X_{\alpha }\in \mathfrak{g}_{\alpha }$ such that $(X_{\alpha },X_{-\alpha }) =1$, $[X_{\alpha },X_{\beta }] = m_{\alpha,\beta }X_{\alpha +\beta }$, where $m_{\alpha ,\beta }\in \mathbb{R}$ satisfying $m_{-\alpha ,-\beta }=-m_{\alpha ,\beta }$ and $m_{\alpha ,\beta }=0$ if $\alpha +\beta \notin \Pi$ (see Helgason \cite{hel} for details).

Given $\alpha \in \mathfrak{h}^{*}$, define $H_{\alpha }$ by $\alpha (\cdot )=( H_{\alpha },\cdot )$ (remember the Cartan-Killing form is nondegenerate on $\mathfrak{h}$) and denote $\mathfrak{h}_{\R}$ the real subspace generated by $H_{\alpha }$, $\alpha \in \Pi $. In the same way, $\mathfrak{h}_{\mathbb{R}}^{*}$ denotes the  real subspace of the dual $\mathfrak{g}^{*}$ generated by the roots.

Denote by $\Pi ^{+}$ the set of positive roots and $\Sigma $ set of simple roots. If $\Theta $ is a subset of simple roots, denote by $\langle
\Theta \rangle $ the set of roots generated by $\Theta $ and $\langle
\Theta \rangle ^{\pm }=\langle \Theta \rangle \cap \Pi ^{\pm }$. Therefore, 
\begin{equation*}
\mathfrak{g}=\mathfrak{h}\oplus \sum_{\alpha \in {\langle \Theta \rangle }%
^{+}}\mathfrak{g}_{\alpha }\oplus \sum_{\alpha \in {\langle \Theta \rangle }%
^{+}}\mathfrak{g}_{-\alpha }\oplus \sum_{\beta \in \Pi ^{+}\setminus {%
\langle \Theta \rangle }}\mathfrak{g}_{\beta }\oplus \sum_{\beta \in \Pi
^{+}\setminus {\langle \Theta \rangle }}\mathfrak{g}_{-\beta }.
\end{equation*}

The parabolic sub-algebra determined by $\Theta $ is given by 

\begin{equation*}
\mathfrak{p}_{\Theta }=\mathfrak{h}\oplus \sum_{\alpha \in {\langle \Theta
\rangle }^{-}}\mathfrak{g}_{\alpha }\oplus \sum_{\alpha \in {\Pi }^{+}}%
\mathfrak{g}_{\alpha }
\end{equation*}

Define  
\begin{equation*}
\mathfrak{q}_{\Theta }=\sum_{\beta \in \Pi ^{+}\setminus {\langle \Theta
\rangle }}\mathfrak{g}_{-\beta },
\end{equation*}
and therefore $\mathfrak{g}=\mathfrak{q}_{\Theta }\oplus \mathfrak{p}_{\Theta }$.

The {\em generalized flag manifold} $\mathbb{F}_{\Theta }$ (associated to $%
\mathfrak{p}_{\Theta }$) is the homogeneous space 
\begin{equation*}
\flag = \dfrac{G}{P_\Theta},
\end{equation*}
where the subgroup $P_{\Theta }$ is the normalizer of $\mathfrak{p}_{\Theta }$ in $G$. 

Recall the compact real form of $\mathfrak{g}$ is the real subalgebra given by 
\begin{equation*}
\mathfrak{u}=\mathrm{span}_{\mathbb{R}}\{i\mathfrak{h}_{\mathbb{R}%
},A_{\alpha },iS_{\alpha }:\alpha \in \Pi \},
\end{equation*}
where $A_{\alpha }=X_{\alpha }-X_{-\alpha }$ and $S_{\alpha }=X_{\alpha
}+X_{-\alpha }$.


Let $U=\exp \mathfrak{u}$ be the corresponding compact
real form of $G$ and put $K_{\Theta }=P_{\Theta }\cap U$. The Lie group $U$ acts transitively on the generalized flag manifold $\mathbb{F}_{\Theta }$ with isotropy subgroup $K_{\Theta }$. Therefore we have also $\mathbb{F}_{\Theta }=U/K_{\Theta }$


Let $\mathfrak{k}_{\Theta }$ be the Lie algebra of $K_{\Theta }$ and denote by $\mathfrak{k}_{\Theta }^{\mathbb{C}}$ its complexification. Thus, $\mathfrak{k}_{\Theta }=\mathfrak{u}\cap \mathfrak{p}_{\Theta }$ and 

\begin{equation*}
\mathfrak{k}_{\Theta }^{\mathbb{C}}=\mathfrak{h}\oplus \sum_{\alpha \in
\langle \Theta \rangle }\mathfrak{g}_{\alpha }. 
\end{equation*}

Let $o=eK_{\Theta}$ be the origin (trivial coset) of $\mathbb{F}_{\Theta }$. Then the tangent
space $T_{o}\mathbb{F}_{\Theta }$ identifies with the orthogonal
complement of $\mathfrak{k}_{\Theta }$ in $\mathfrak{u}$, that is, 
\begin{equation*}
T_{o}\mathbb{F}_{\Theta }= \mathfrak{m}_{\Theta }=\mathrm{span}_{\mathbb{R}%
}\{A_{\alpha },iS_{\alpha }:\alpha \notin \langle \Theta \rangle
\}=\sum_{\alpha \in \Pi \setminus \langle \Theta \rangle }\mathfrak{u}%
_{\alpha }, 
\end{equation*}
where $\mathfrak{u}_{\alpha }=\left( \mathfrak{g}_{\alpha }\oplus \mathfrak{g%
}_{-\alpha }\right) \cap \mathfrak{u}=\mathrm{span}_{\mathbb{R}}\{A_{\alpha
},iS_{\alpha }\}$.
By complexifying $\mathfrak{m}_{\Theta }$, we obtain the
complex tangent space of $T_{o}^{\mathbb{C}}\mathbb{F}_{\Theta }$, which can be identified with 
\begin{equation*}
\mathfrak{m}_{\Theta}^{\mathbb{C}}=\mathfrak{q}_{\Theta }=\sum_{\beta \in
\Pi \setminus \langle \Theta \rangle }\mathfrak{g}_{\beta }. 
\end{equation*}

\subsection{Almost complex structures}

An $\Kt$-invariant almost complex structure $J$ on $\flag$ is completely determined by its value at the origin, that is, by $J: \mt \to \mt$ in the tangent space of $\flag$ at the origin. The map $J$ satisfies $J^2=-1$ and commutes with the adjoint action of $\Kt$ on $\mt$. We also denote by $J$ its complexification to $\qt$.

The invariance of $J$ entails that $J(g_{\sigma}) = g_{\sigma}$ for all $\sigma \in \piteta$. The eigenvalues of $J$ are $\pm \sqrt{-1}$ and the eigenvectors in $\qt$ are $X_{\alpha}$, $\alpha \in \Pi\setminus\langle \Theta \rangle$. Hence, in each irreducible component, $J= \sqrt{-1}\epsilon_\sigma Id$, where $\epsilon_\sigma = \pm 1$ and $\epsilon_{-\sigma} = -\epsilon_{\sigma}$.
An $U$-invariant structure on $\flag$ is then completely described by a set of signs $\{\epsilon_{\sigma}\}_{\sigma \in \piteta}$ with $\epsilon_{-\sigma} = -\epsilon_{\sigma}$.

The eigenvectors associated to $\sqrt{-1}$ are of type $(1,0)$ while the eigenvectors associated to $-\sqrt{-1}$ are of type $(0,1)$. Thus, the $(1,0)$ vectors at the origin are multiples of $X_{\alpha}$, $\epsilon_{\sigma} = 1$, and the $(0,1)$ vectors are also multiples of $X_\alpha$, $\epsilon_\sigma = -1$, where $\alpha \in \sigma$. Also,

\begin{equation}
\label{et01}
T_x\flag^{(1,0)} = \sum_{\sigma \in \piteta^+}E_{\epsilon_\sigma \sigma} \ \ \ \ \ \ T_x\flag^{0,1} = \sum_{\sigma \in \piteta^+} E_{\epsilon_{-\sigma}\sigma}.
\end{equation}

Since $\flag$ is a homogeneous space of a complex Lie group, it has a natural structure of a complex manifold. The associated integrable almost complex structure $J_c$ is given by $\epsilon_\sigma = 1$ if the roots in $\sigma$ are all negative. The conjugate structure $-J_c$ is also integrable.  

\subsection{Isotropy representation}

The adjoint representations of $\kt$ and $\Kt$ leave $\mt$ invariant, so that we get a well-defined representation of both $\kt$ and $\Kt$ in $\mt$. Analogously, the complex tangent space $\qt$ is invariant under the adjoint representation of $\kt^\C$ and we can define the complexification of the isotropy representation from $\kt^\C$ to Aut($\mt^\C$). Since the representation is semissimple, we can decompose it into irreducible components, where each irreducible component is a sum of root spaces.

We will denote an irreducible component of $\mt^{\C} = \qt$ by $\gs$, where $\sigma$ is the subset of roots such that
$$\gs = \sum_{\alpha \in \sigma} \gal,$$
and we write $\piteta$ for the set of $\sigma$'s. Then, we have
$$\qt = \sum_{\sigma \in \piteta} \gs.$$

The roots in each irreducible component $\sigma \in \piteta$ are either all positive or all negative, so we write $\piteta^+$ and $\piteta^-$ for the set of those irreducible components containing only positive roots and negative roots, respectively.

Denote by $\Sigma(\Theta)$ the set of $\sigma \in \piteta$ that has height $1$ module $\langle \Theta \rangle$, i.e,

$$\Sigma(\Theta) = \{\sigma \in \piteta: \text{ the height of } \sigma \in \piteta \text{ is } 1\}.$$

Since $\Ad(k)(\gs) = \gs$, for each $\sigma \in \piteta$, we have a well defined complex plane field on $\flag$ given by
$$E_\sigma (k \cdot o) = k_*(\gs)$$
and, for any $x \in \flag$, we have
$$T_x^{\C}\flag = \sum_{\sigma \in \piteta}E_\sigma(x).$$

\subsection{Invarian metrics}
There is a 1-1 correspondence between $U$-invariant metrics $g$ on $\flag$ and $\Ad(\Kt)$-invariant scalar products $B$ on $\mt$ (see for instance \cite{KoNo}). Any $B$ can be written as $$B(X,Y) = \langle X,Y \rangle_\Lambda =  -(\Lambda X,Y),$$ with $X,Y \in \mt$, where $\Lambda$ is an $\Ad(\Kt)$-invariant positive symmetric operator on $\mt$ with respect to the Cartan-Killing form. The scalar product $B(.,.) = \langle X,Y \rangle_\Lambda$ admits a natural extension to a symmetric bilinear form on $\mt^\C = \qt$. We will use the same notation for this extension.

As a consequence of Schur's Lemma, in each irreducible component of $\qt$ we have $\Lambda = \lambda_\sigma$, with $\lambda_{-\sigma} = \lambda_\sigma > 0$. 

\begin{remark}
In the next sections we abuse notation and denote an invariant metric $g$ on $\flag$ just by $\Lambda= (\lambda_\sigma)_{\sigma \in \piteta}$, that is, a $n$-uple of positive real numbers indexed by the irreducible components of $\mathfrak{m}_\Theta^\C = \qt$.
\end{remark}

\section{Stability of equiharmonic maps on generalized flag manifolds}
\subsection{$J$-holomorphic curves on $\flag$}
If $M = M^2$ is a Riemann surface and $\phi: M \to \flag$ is a differentiable map, we let $d^{\C}\phi$ be the complexification of the differential of $\phi$. We endow $\flag$ with a complex structure $J$ and, as usual, decompose $d^{\C}\phi$ into $\frac{\partial \phi}{\partial z}(p): TM^{(1,0)} \to T\flag^{(1,0)}$ and $\frac{\partial \phi}{\partial \overline{z}}: TM^{(0,1)} \to T\flag^{(0,1)}$, which are identified with vectors in the complex tangent space. We use the decomposition of $T^{\C}\flag$ into irreducible components. By $(\ref{et01})$, we have
\begin{equation*}
\dfrac{\partial \phi}{\partial z}(p) = \sum_{\sigma \in \piteta^+}\phi_{\epsilon_\sigma\sigma}(p) \ \ \ \ \ \ \dfrac{\partial \phi}{\partial \overline{z}}(p) = \sum_{\sigma \in \piteta^+}\phi_{\epsilon_{-\sigma}\sigma}(p),
\end{equation*}
where, for each $\sigma \in \piteta$, the function $\phi_\sigma: M^2 \to E_\sigma$ takes value in $E_\sigma(\phi(p))$, $p \in M$.

Given an almost complex structure on $\flag$, a map $\phi: M^2 \to \flag$ is $J$-holomorphic if, for all $p \in M$, holds
\begin{equation*}
\dfrac{\partial \phi}{\partial \overline{z}}(p) = \sum_{\sigma \in \piteta^+}\phi_{\epsilon_{-\sigma}\sigma}(p) = 0.
\end{equation*}

\subsection{Stability of equiharmonic maps on $\flag$.}

In \cite{NeGrSM11} were proved  several results about stability and non-stability of equiharmonic maps (maps that are harmonic with respect to any invariant metric) in a generalized flag manifold $\flag$. 

Let us recall some of the results in \cite{NeGrSM11}. Consider $M^2$ a compact Riemann surface equipped with a metric $g$, let $(N,h)$ be a compact Riemannian manifold and $\phi:(M^2,g)\to (N,h)$ a differentiable map. The {\em energy} of $\phi$ is given by 
$$
E(\phi)=\frac{1}{2}\int_{M}{|d\phi|^2\, \omega_g},
$$ 
where $\omega_g$ is the volume measure defined by the metric $g$ and $|d\phi|$ is the Hilbert-Schmidt norm of $d\phi$. The differentiable map $\phi$ is {\em harmonic} if it is a critical point of the energy functional. 

Let us restrict ourselves to harmonic maps from compact Riemann surfaces to a generalized flag manifold $\flag$. Given a harmonic map $\phi: (M^2,g) \to (\flag, ds^2_\Lambda)$, we consider perturbed maps $\phi^t(p)$ given by 
$$\phi^t(p) = \exp(tq(p)) \cdot \phi(p),$$where $q: M \to \ul$ is a smooth map. The {\em second variation of the energy of} $\phi$, denoted by $\sve$, is given by
$$
I_{\Lambda }^{\phi }(q)=\left. \frac{d^{2}}{dt^{2}}\right|
_{t=0}\!\!\!\!\!E(\phi ^{t}).
$$


\begin{definition}
Let $\phi: (M^2,g) \to (\flag, ds^2_\Lambda)$ be an arbitrary harmonic map. We say that $\phi$ is {\em stable} if $I_{\Lambda}^\phi(q) \geq 0$ for any variation $q: M^2 \to \g$. Otherwise, we say that $\phi$ is {\em unstable}.
\end{definition}

We are interested in the following situation: let $\Lambda_0$ be a invariant metric on $\flag$ and $\Lambda_1$ another invariant metric obtained from $\Lambda_0$ by a special kind of perturbation (defined bellow). Suppose that the map $\phi:M^2\to \flag$ is harmonic with respect to {\em both} invariant metrics $\Lambda_0$ and $\Lambda_1$. One of the main contribution of \cite{NeGrSM11} is provide the understanding of the behavior of $I^\phi_{\Lambda_1}$ in therms of $I^\phi_{\Lambda_0}$.

\begin{definition} Let $\Pp$ be a subset of $\piteta$. An invariant metric $\Lambda^\# = (\lambda_\sigma^\#)_{\sigma \in \piteta}$ is called a {\em $\Pp$-perturbation of $\Lambda$} if the following holds:
\begin{enumerate}
\item $\lambda_\sigma^\# = \lambda_\sigma$ for all $\sigma \in \Pp$;
\item $\lambda_\sigma^\# = \lambda_\sigma + \xi_\sigma > 0$, $\xi_\sigma \in \R$,  for all $\sigma \in \piteta\setminus\Pp$. 
\end{enumerate}
\end{definition}

Since we need to consider maps that are harmonic with respect to the invariant metric $\Lambda$ and the perturbed metric $\Lamjv$, it is natural to consider equiharmonic maps.

\begin{definition}
A map $\phi:M^2\to \flag$ is called {\em equihamonic map} if it is harmonic with respect {any} invariant metric on $\flag$. 
\end{definition}

Examples of equiharmonic maps are the so called generalized holomorphic-horizontal maps, whose definition is given bellow.


\begin{definition}
A map $\phi: M^2 \to (\flag, J)$ is called generalized holomorphic-horizontal if it is $J$-holomorphic and satisfies $\phi_\sigma = 0$ if $\sigma \in \piteta \setminus \Sigma(\Theta)$.
\end{definition}

\begin{remark}
Here we are using the same terminology of Bryant \cite{Bry85}. In \cite{BR90}, these maps are called \textit{super-horizontal maps}.
\end{remark}

Observe that, when working with generalized holomorphic-horizontal maps, we are taking $\Pp = \Sigma(\Theta)$. Also, the following result guarantees that those maps are equiharmonic maps. 

\begin{theorem} \cite{GNSM14}
If $\phi$ is a generalized holomorphic-horizontal map, then $\phi$ is equiharmonic.
\end{theorem}

%

The following theorems are crucial in the understanding of the stability of generalized holomorphic maps under the Ricci flow of a perturbed invariant metric. We start with a classical result of Lichnerowicz:

\begin{theorem} \cite{Lich70}
\label{kahler}
Let $\phi: (M^2, J, g) \to (\flag, J, ds_{\Lambda}^2)$ be a $J$-holomorphic map, where $(\flag, J, ds_{\Lambda}^2)$ is a K\"{a}hler structure. Then, $\phi$ is harmonic and stable.
\end{theorem}

Now we consider some special $\mathcal{P}$-perturbation of a invariant K\"ahler structure on $\flag$ in order to construct several examples of stable/unstable harmonic maps.

\begin{theorem} \cite{NeGrSM11}
\label{kahlersoma}
Let $\phi: M^2 \to \flag$ be a generalized holomorphic-horizontal map and $(\flag, J, ds_{\Lambda}^2)$ a K\"{a}hler structure. Consider a $\Pp$-perturbation $\Lamjv$ of this structure, with $\xi_\sigma \geq 0$ for all $\sigma \in \piteta \setminus \Pp$. Then $\phi: (M^2,g) \to (\flag, ds_{\Lamjv}^2)$ is stable.
\end{theorem}

\begin{theorem} \cite{NeGrSM11}
\label{kahlersubtracao}
Let $\phi: M^2 \to \flag$ be a generalized holomorphic-horizontal map and take a K\"{a}hler structure $(\flag, J, ds_{\Lambda}^2)$ with $\Lambda = (\lambda_\sigma)_{\sigma \in \piteta}$. Suppose that $\Lamjv = (\lamjv)_{\sigma \in \piteta}$ is another metric such that, for some $\sigma_0 \in \piteta \setminus \Sigma(\Theta)$, the inequality $\xi_{\sigma_0} = \lambda_{\sigma_0}^\# - \lambda_{\sigma_0} < 0$ holds. Then, $\phi$ is unstable with respect to $\Lamjv$. 
\end{theorem}

\section{Stability on $\flag$ under the Ricci Flow}

In this section we study the stability of a generalized holomorphic-horizontal map $\phi: M^2 \to \flag$ under the homogeneous Ricci flow of a perturbed invariant metric for three types of flag manifolds: $\text{SO}(2n+1)/(\text{U}(k)\times \text{SO}(2n+1))$, $\text{Sp}(n)/(\text{U}(m) \times \text{Sp}(k))$ and $\ft$. For more details about this topic, see \cite{GrMa09} and \cite{GrMa12}.

\subsection{Homogeneous Ricci Flow}

We will begin by reviewing the global behaviour of the homogeneous Ricci flow on $\text{SO}(2n+1)/(\text{U}(k)\times \text{SO}(2n+1))$, $\text{Sp}(n)/(\text{U}(m) \times \text{Sp}(k))$ and $\ft$. 

\subsubsection{Ricci flow on $\tipoum$ and $\tipodois$}

The isotropy representation of the families $\tipoum$ and $\tipodois$ decompose into two non-equivalent isotropy summands, that is, $\mt = \m_1 \oplus \m_2$. We keep our notation and denote an invariant metric just by $\Lambda=(\lambda_1,\lambda_2)$.  The Ricci tensor of an invariant metric $\Lambda$ is again an invariant tensor, and therefore completely determined by its value at the origin of the homogeneous space and constant in each irreducible component of the isotropy representation. In the case of $\text{SO}(2n+1)/(\text{U}(k) \times \text{SO}(2n+1))$ and $\text{Sp}(n)/(\text{U}(m) \times \text{Sp}(k))$, the components of the Ricci tensor are given, respectively, by

\begin{equation*}
\begin{cases} r_1 = -\dfrac{2(m-1)}{2n-1} - \dfrac{1+2k}{2(2n-1)}\dfrac{\lambda_1^2}{\lambda_2^2}, \\
\\
r_2 = -\dfrac{n+k}{2n-1} - \dfrac{m-1}{2(2n-1)}\dfrac{(\lambda_2^2-(\lambda_1 - \lambda_2)^2)}{\lambda_1\lambda_2},
\end{cases}
\end{equation*}
and
\begin{equation*}
\begin{cases} r_1 = -\dfrac{2+2m)}{2n+2} - \dfrac{2k}{4n+4}\dfrac{\lambda_1^2}{\lambda_2^2}, \\
\\
r_2 = -\dfrac{4m+4k+3}{4n+4} + \dfrac{4m+2}{16n+16}\dfrac{\lambda_1}{\lambda_2}.
\end{cases}
\end{equation*}

The Ricci flow equation on the manifold $M$ is defined by
\begin{equation}\label{eqnricci}
\frac{\partial g(t)}{\partial t}=-2Ric(g(t)),
\end{equation}
where $Ric(g)$ is the Ricci tensor of the Riemannian metric $g$. The solution of this equation, the so called Ricci flow, is a $1$-parameter family of metrics $g(t)$ in $M$.

The homogeneous Ricci flow equation for invariant metrics is given by the following systems of ODEs: 

\begin{equation}
\label{riccitype1}
\begin{cases} \dot{x} = \dfrac{2(m-1)}{2n-1} + \dfrac{1+2k}{2(2n-1)}\dfrac{x^2}{y^2}, \\
\\
\dot{y} = \dfrac{n+k}{2n-1} + \dfrac{m-1}{2(2n-1)}\dfrac{(y^2-(x-y)^2)}{xy},
\end{cases}
\end{equation}
for
$$\dfrac{\text{SO}(2n+1)}{\text{U}(k)\times\text{SO}(2m+1)}, \quad n=m+k, \ m>1 \text{ and } k \neq 1,$$
and
\begin{equation}
\label{riccitype2}
\begin{cases} \dot{x} = \dfrac{2m+2)}{2n+2} + \dfrac{2k}{4n+4}\dfrac{x^2}{y^2}, \\
\\
\dot{y} = \dfrac{4m+4k+3}{4n+4} - \dfrac{4m+2}{16n+16}\dfrac{x}{y},
\end{cases}
\end{equation}
for
$$\dfrac{\text{Sp}(n)}{\text{U}(m)\times\text{Sp}(k)}, \quad n=m+k, \ m\geq 1 \text{ and } k \geq 1.$$

In \cite{GrMa12}, using the Poincar\'{e} compactification on the space of invariant Riemannian metrics, the dynamics of (\ref{riccitype1}) and (\ref{riccitype2}) was completely described as follows, respectively.

$$
\gamma_1(t) = \left(\dfrac{2(m-1)}{m+2k}t,t\right), \quad
\gamma_2(t)=\left(2t,t\right), \,\,\, \text{ in the case } (\ref{riccitype1})
$$
and 
$$
\gamma_1(t) = \left(t,\dfrac{1}{4}\dfrac{4k+2m+1}{m+1}t\right), \quad
\gamma_2(t)=\left(t,\dfrac{1}{2}t\right), \,\,\, \text{ in the case } (\ref{riccitype2}).
$$

The global behavior of the Ricci flow on both generalized flag manifolds is described using its phase portrait (see Figure \ref{retratofase}). 

\begin{figure}[!htb]
\label{retratofase}
\centering
\includegraphics[scale=1]{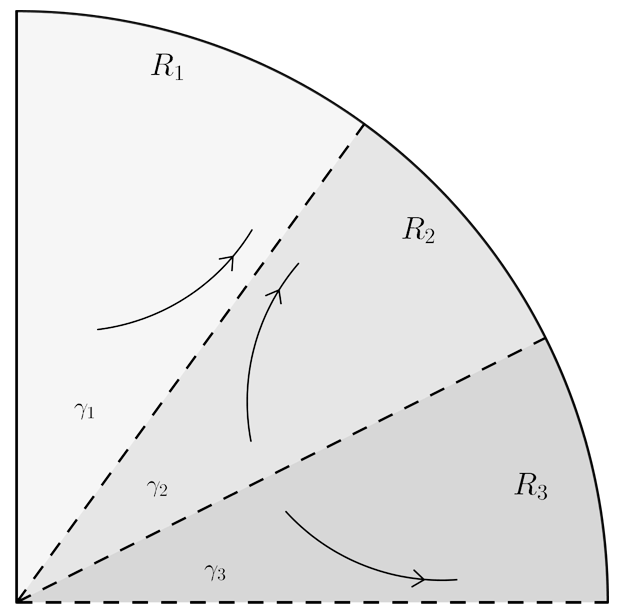}
\caption{Phase portrait of the Ricci flow on \hspace{\linewidth}
$\text{SO}(2n+1)/(\text{U}(k) \times \text{SO}(2n+1))$ and on $\text{Sp}(n)/(\text{U}(m) \times \text{Sp}(k))$.}
\end{figure}

One can describe precisely the ``asymptotic behavior'' of the flows line of the Ricci flow. Let $g_0$ be an invariant metric and $g(t)$ the Ricci flow with initial condition $g_0$. We will denote by $g_\infty$ the limit $\lim_{t\to \infty} g(t)$. 
\begin{theorem}\label{teoricci}{\cite{GrMa12}} 
Let $g_0$ be an invariant metric on $\text{SO}(2n+1)/(\text{U}(k) \times \text{SO}(2n+1))$ or $\text{Sp}(n)/(\text{U}(m) \times \text{Sp}(k))$ and $R_1$, $R_2$, $R_3$, $\gamma_1$ and $\gamma_2$ as described in Figure \ref{retratofase}. We have:
\begin{enumerate}
\item[a)] if $g_0 \in R_1\cup R_2 \cup \gamma_1$ then $g_{\infty}$ is the Einsten (non-K\"ahler) metric;
\item[b)] if $g_0 \in \gamma_2$ then $g_{\infty}$ is the K\"ahler-Einstein metric.
\item[c)] if $g_0 \in R_3$, consider the natural fibration from a flag manifold in a symmetric space $G/K \to G/H$. Then, the Ricci flow $g(t)$ with $g(0) = g_0$ evolves in such a way that the diameter of the base of this fibration converges to zero when $t \to \infty$.
\end{enumerate}
\end{theorem}

\subsubsection{Ricci flow on the space $\ft$}

The isotropy representation of the space $\ft$ decomposes into three non-equivalent irreducible components, that is, $\mt = \m_1 \oplus \m_2 \oplus \m_3$. So, we keep our notation and denote an invariant metric by $(\lambda_{12}, \lambda_{13}, \lambda_{23})$. In the case of the space $\ft$, the components of the Ricci tensor are given by
\begin{equation}
\label{riccift}
\begin{cases} r_{12} = \dfrac{1}{2\lambda_{12}} + \dfrac{1}{12}\left(\dfrac{\lambda_{12}}{\lambda_{13}\lambda_{23}} - \dfrac{\lambda_{13}}{\lambda_{12}\lambda_{23}} - \dfrac{\lambda_{23}}{\lambda_{12}\lambda_{13}}\right), \\
\\
r_{13} = \dfrac{1}{2\lambda_{13}} + \dfrac{1}{12}\left(\dfrac{\lambda_{13}}{\lambda_{12}\lambda_{23}} - \dfrac{\lambda_{12}}{\lambda_{13}\lambda_{23}} - \dfrac{\lambda_{23}}{\lambda_{12}\lambda_{13}}\right), \\
\\
r_{23} = \dfrac{1}{2\lambda_{23}} + \dfrac{1}{12}\left(\dfrac{\lambda_{23}}{\lambda_{12}\lambda_{13}} - \dfrac{\lambda_{13}}{\lambda_{23}\lambda_{12}} - \dfrac{\lambda_{12}}{\lambda_{23}\lambda_{13}}\right).
\end{cases}
\end{equation}

The Ricci flow equation system for a left-invariant metric on $\ft$ is given by
\begin{equation}\label{ricciF3}
\dot{\lambda}_{ij} = -2r_{ij}, \qquad 1 \leq i<j\leq 3.
\end{equation}

The invariant lines of (\ref{ricciF3}) can be described as follows: define $\gamma_j(t) := tp'_j$, where

$p'_1 = (1/\rho_1, (1/\rho_1)(2+2\sqrt{2}), 1/\rho_1),$

$p'_2 = (1/\rho_2,1/\rho_2,1/\rho_2),$

$p'_3 = (\rho_3(-1/2 + \sqrt{2}/2),\rho_3(-1/2 + \sqrt{2}/2),\rho_3),$

$p'_4 = ((1/\rho_1)(2+2\sqrt{2}),1/\rho_1,1/\rho_1),$

$1/\rho_3 = \sqrt{1+2(-1/2+1/2\sqrt{2})^2}$, $\rho_1 = \sqrt{2+(2+\sqrt{2})^2}$ and $\rho_3 = \sqrt{3}$. 

We remark that $\gamma_1, \gamma_2, \gamma_3$ and $\gamma_4$ are solutions for (\ref{riccift}). It is well know that the manifold $SU(3)/T^2$ admits four invariant Einstein metrics: three are K\"ahler-Einstein metrics, represented by $\gamma_1$, $\gamma_3$, $\gamma_4$, and the normal Einstein metric (non-K\"ahler), represented by $\gamma_2$.

Using techniques from dynamical systems (Poincar\'e compactifications, Lyapunov exponents), the behavior of the Ricci flow near the normal Einstein metric is described as follows.
\begin{theorem} \cite{GrMa09}
\label{thricci}
Consider $\varepsilon > 0$ sufficiently small and $\Lambda = (\lambda_{12}, \lambda_{13},\lambda_{23})$ with $\lambda_{ij} > 0$, $\parallel \Lambda \parallel > \delta$ for $\delta > 1/2$ and $d(\Lambda,\gamma_2) < \varepsilon$. Let $g_t$ be the Ricci flow with $g_0$ the metric defined by $(\lambda_{12}, \lambda_{13},\lambda_{23})$. Then $g_\infty$ is a normal (Einstein) metric. In particular, if $\Lambda \notin \gamma_2$, then $g_0$ is left-invariant and $g_\infty$ is bi-invariant.
\end{theorem}
\subsection{Stability on $\tipoum$, $\tipodois$ and $\ft$ under the homogeneous Ricci flow.}

Now we are going to establish how the stability of a generalized holomorphic-horizontal map $\phi: M^2 \to \flag$ behaves under the Ricci flow of a perturbed invariant metric of $\tipoum$ and $\tipodois$. In the next Theorem, $\flag$ will denote one of those two types of flags. 

\begin{theorem}
Let $\phi: M^2 \to \flag$ be a generalized holomorphic-horizontal map and $(\flag,J,g_0)$ a K\"{a}hler structure with $g_0 = \Lambda = (\lambda_1,\lambda_2)$. Suppose that $g_0^\# = \Lamjv = (\lambda_1, \lambda_2^\#)$ is any $\Pp$-perturbation of $\Lambda$. The following holds:

\begin{enumerate}
\item If $0 < \lambda_2^\# < \lambda_2$, then $\phi$ is unstable with respect to $g_0^\#$ and remains unstable under the Ricci flow, $g_t^\#$, with initial condition $g_0^\#$; 
\item If $0 < \lambda_2 < \lambda_2^\#$, then $\phi$ is stable with respect to $g_0^\#$ and remains stable under the Ricci flow, $g_t^\#$, with initial condition $g_0^\#$ and for $t<\infty$.
\end{enumerate}
\end{theorem}
\begin{proof}
Borel \cite{Borel54} described the K\"{a}ler structures on a generalized flag manifold $\flag$. In fact, $(\flag,J,g_0)$ is K\"{a}hler if and only if $J$ is integrable and, if $\alpha \in \sigma \in \piteta$ can be written as
$$\alpha = \sum_{i=1}^k n_i\alpha_i$$
with $\alpha_i \in \sigma_i \in \Sigma(\Theta)$, then $\lambda_\alpha = \disp\sum_{i=1}^k n_i \alpha_i$, where $n_i \geq 0$ if $\alpha$ is positive.

In our case, both $\tipoum$ and $\tipodois$ have only two isotropy summands, i.e., $\piteta = \{\sigma_1, \sigma_2\}$. Here, $\sigma_1$ is the set of roots whose height is one module $\langle \Theta \rangle$, that is, $\Sigma(\Theta) = {\sigma_1}$. If $\alpha \in \sigma_2$, then $\alpha$ can be written as $\alpha = \alpha_1 + \alpha_2$, where $\alpha_1, \alpha_2 \in \sigma_1$. Then, $\lambda_{\sigma_2} = 2\lambda_{\sigma_1}$ and $g = \Lambda = (1,2)$ is the only K\"{a}hler metric in $\flag$ (up to scale).

Let $g_0 = (\lambda_1,2\lambda_1)$.  Theorem \ref{kahler} states that $\phi$ is stable with respect to $g_0$ and, by theorem \ref{kahlersoma}, we have that $\phi$ is also stable with respect to the perturbed metric $g_{0,1}^\# = (\lambda_1, 2\lambda_1 + \xi_1)$, where $\xi_1 \geq 0$. Also, theorem \ref{kahlersubtracao} says that $\phi$ is unstable with respect to the perturbed metric $g_{0,2}^\# = (\lambda_1,2\lambda_1 - \xi_2)$, where $\xi_2 < 0$ and $2\lambda_1 - \xi_2 > 0$. Observe that, according do figure \ref{retratofase}, $g_{0,1}^\# \in R_3 \cup \gamma_2$ and $g_{0,2}^\# \in R_1 \cup R_2 \cup \gamma_1$. If $g_{t,1}^\#$ and $g_{t,2}^\#$ are the homogeneous Ricci flow for the metrics $g_{0,1}^\#$ and $g_{0,2}^\#$, respectively, then $\phi$ is stable with respect to $g_{t,2}^\#$ for finite $t$ and is unstable with respect to $g_{t,1}^\#$ for all $t$.
\end{proof}

\begin{corollary}
The homogeneous Ricci flow preserves the stability/non-stability of a generalized holomorphic-horizontal map on the homogeneous spaces $\tipoum$ and $\tipodois$.
\end{corollary}

Let us consider now the homogeneous space $SU(3)/T^2$ in order to produce an example where the Ricci flow do not preserve the stability of harmonic maps. The following Lemma tell us about the non-stability of the normal metric on full flag manifolds.

\begin{lemma}{\cite{NeGrSM11}}\label{normal-stab}
Consider $(\ft,g)$ equipped with the normal metric. Let $\phi: M^2 \to \ft$ be an arbitrary generalized holomorphic-horizontal map. Then, $\phi$ is unstable with respect to $g$.
\end{lemma}

%
%

We have that the complexification of $\su(3)$ is the Lie algebra $\sll(3,\C)$. The root space decomposition of $\sll(3,\C)$ is given as follows. Consider the Cartan subalgebra $\h$ given by diagonal matrices of trace zero. Then, the root system of $\sll(3,\C)$ relative to $\h$ is composed by $\alpha_{ij} := \varepsilon_i - \varepsilon_j$, $1 \leq i \neq j \leq 3$, where $\varepsilon_i$ is the functional given by $\varepsilon_i: diag\{a_1,a_2,a_3\} \to a_i$. A simple system of roots is $\Sigma = \{\alpha_{12}, \alpha_{23}\}$. 

Recall the $SU(3)$-invariant Einstein metrics are described as follow: the K\"ahler-Einstein metrics parametrized by $\Lambda_1=(1,1,2)$, $\Lambda_2=(1,2,1)$, $\Lambda_3=(2,1,1)$; the normal Einstein (non-K\"ahler) metric parametrized by $\Lambda_4=(1,1,1)$.  

In the proposition bellow, we choose a $\mathcal{P}$-perturbation being $\mathcal{P}=\{\alpha_{12} \}$ and show that the {\em equiharmonic} 2-sphere $\phi:S^2\to \ft$ subordinated to $\mathcal{P}$ is unstable under the Ricci flow.    

\begin{theorem}\label{princ2}
Let $\phi: S^2 \to \ft$ be a $J$-holomorphic map subordinated to $\mathcal{P}=\{ \alpha_{12} \}$ and $(\ft,J,g_0)$ a K\"{a}hler-Einstein structure with $g_0 = \Lambda = (2,1,1)$. Consider $\varepsilon$ sufficiently small and a $\Pp$-perturbation of $g_0$ given by $g_0^\# = \Lamjv = (2, 2-\varepsilon, 2-\varepsilon)$. Denote by $g_t^\#$ the Ricci flow with initial condition $g_0^\#$. Then, $\phi$ is stable with respect to $g_0^\#$ and is unstable with respect to $g^\#_\infty$.
\end{theorem}

\begin{proof}
Since $g_0$ is a K\"{a}hler-Einstein metric, $\phi$ is stable with respect to $g_0$, by Theorem \ref{kahler}. We also remark  (using Theorem \ref{kahlersoma}) that $\phi$ is stable with respect to the $\Pp$-perturbed metric $g_0^\#$. 

Also, $\parallel \Lambda \parallel > 1/2$. The invariant line represented by the normal metric on the  phase portrait of the Ricci flow  is given by $$\gamma_2(t) = (t/\sqrt{3},t/\sqrt{3},t/\sqrt{3}).$$ Then, $$d(g_0^\#, \gamma_2) < \varepsilon.$$ By Theorem \ref{thricci}, $g^\#_\infty$ is a normal Einstein and the result follows from Lemma \ref{normal-stab}.     

\end{proof}

\begin{remark}
According to \cite{BR90} the harmonic 2-sphere subordinated to a simple root $\{\alpha_{12} \}$ described in Theorem \ref{princ2} is a generator of the second homotopy group $\pi_2(SU(3)/T^2)\approx \mathbb{Z}\oplus \mathbb{Z}$ (each generator is represented by a simple root of $\mathfrak{sl}(3,\mathbb{C})$).
\end{remark}
%

\end{document}